\documentclass[12 pt]{amsart}
\usepackage{amsmath,amssymb,amsthm}
\usepackage{longtable}
\usepackage{a4wide}
\usepackage{color}

\usepackage{tikz}

\newtheorem{thm}{Theorem}[section]
\newtheorem*{thm*}{Theorem}
\newtheorem*{conj*}{Conjecture}
\newtheorem{cor}[thm]{Corollary}

\newtheorem{lem}[thm]{Lemma}
\newtheorem{prop}[thm]{Proposition}
\newtheorem{eg}[thm]{Example}
\theoremstyle{remark}

\theoremstyle{definition}
\newtheorem{defn}[thm]{Definition}
\newcounter{claim}[thm]

\newcommand{\sym}[1]{\mathrm{Sym}(#1)}
\newcommand{\alt}[1]{\mathrm{Alt}(#1)}

\newcommand{\grp}[1]{\langle #1 \rangle}				%group generated by
\newcommand{\op}[2]{\mathrm{O}_{#1}(#2)}	%p-core for any prime any group
\newcommand{\norm}[2]{\text{N}_{#1}(#2)}				%normaliser
\newcommand{\cent}[2]{\text{C}_{#1}(#2)}				%centraliser
\newcommand{\vst}[2]{G_{#1}^{[ #2]}}		% stab of dist from any vert
\newcommand{\vstx}[1]{G_x^{[ #1 ]}}			%stabilizer of distance from x
\newcommand{\syl}[2]{\mathrm{Syl}_{#1}( #2)}	% sylow subgps
\newcommand{\zent}[1]{\mathrm{Z}(#1)}
\newcommand{\ba}[1]{\overline{#1}}
\newcommand{\aut}[1]{\mathrm{Aut}(#1)}
\newcommand{\cyc}[1]{\mathrm C_{#1}}
\newcommand{\fst}[1]{\mathbf F^*( #1) }

\title{On locally semiprimitive graphs and a theorem of Weiss}

\author{Michael Giudici}
\author{Luke Morgan}

\address{
School of Mathematics and Statistics (M019)\\
University of Western Australia\\
Crawley, 6009\\
Australia} 

\email{michael.giudici@uwa.edu.au}	\email{luke.morgan@uwa.edu.au}

\thanks{This research was supported under Australian Research Council's \emph{Discovery Projects} funding scheme (project number DP 120100446).}

\begin{document}

\begin{abstract}
In this paper we investigate graphs that admit a group acting arc-transitively such that the local action
is semiprimitive with a regular normal nilpotent subgroup. This type of semiprimitive group is a
generalisation of an affine group. We show that if the graph has valency coprime to six, then there is a
bound on the order
of the vertex stabilisers depending on the valency alone.
We also prove a detailed structure theorem for the vertex stabilisers in the
remaining case. This is a contribution to an ongoing project to investigate the validity of the Poto\v{c}nik-Spiga-Verret
Conjecture.
\end{abstract}

\maketitle

\section{introduction}

All  graphs in this paper are finite, connected and simple and every action of a group on a graph is
faithful. If a group $G$ acts on a graph $\Gamma$ and  $x$ is a  vertex of $\Gamma$, we write $\Gamma(x)$
for the neighbourhood of $x$ in $\Gamma$ and $G_x^{\Gamma(x)}$ for the permutation group induced on the set
$\Gamma(x)$ by the stabiliser $G_x$.
If $\mathcal P$ is a property of permutation groups and $L$ is some permutation group, we
say that the pair $(\Gamma,G)$ is \emph{locally} $\mathcal P$, respectively, \emph{locally} $L$ to indicate that
for all vertices $x$ of $\Gamma$,  $G_x^{\Gamma(x)}$ satisfies property $\mathcal P$, respectively,
$G_x^{\Gamma(x)}$ is permutationally isomorphic to $ L$.

Following  \cite{verret},  if there is a constant $c(L)$ such that for every locally $L$ pair $(\Gamma,G)$
we have $|G_x| \leqslant  c(L)$, we say that $L$ is \emph{graph-restrictive}. In this language, the long-standing Weiss Conjecture \cite{weisscon} asserts that primitive permutation groups are graph-restrictive and the Praeger Conjecture asserts that quasiprimitive groups are graph-restrictive. Certain
classes of permutation groups are known to be graph-restrictive: it is easy to see that regular groups are
graph-restrictive for example. However the proof that 2-transitive groups are graph-restrictive  
\cite[Theorem 1.4]{trofweiss1} is a deep result which uses the Classification of Finite Simple Groups. 
% By \cite[Theorem 4]{PSV}  a graph-restrictive group must be \emph{semiprimitive} (see Definition~\ref{defn:sp}).
Fresh light was cast upon this problem by \cite[Theorem 4]{PSV} which shows that a graph-restrictive group is necessarily \emph{semiprimitive} (see Definition~\ref{defn:sp}). Going further, the authors of \cite{PSV} stated the Poto\v{c}nik-Spiga-Verret (PSV) Conjecture:  a permutation group is graph-restrictive if and only if it is semiprimitive.
% The authors also stated the Potoc\v{n}ik-Spiga-Verret Conjecture (\cite[Conjecture 3]{PSV}) (which we will refer to as the PSV Conjecture) that  a permutation group is graph-restrictive if and only if it is semiprimitive.
Since the class
of semiprimitive permutation groups properly encompasses the classes of  primitive and quasiprimitive groups,
the PSV Conjecture is broader in scope than both the Weiss and Praeger Conjectures and a proof of the former
would imply the truth of the latter two conjectures.

In this paper we investigate the validity of the PSV Conjecture for semiprimitive groups with a regular
normal nilpotent subgroup. One can view this type of semiprimitive group as a generalisation of affine
groups, which form one of the eight types of primitive groups. However our knowledge of semiprimitive groups
pales in comparison to our knowledge of primitive groups, there is no  O'Nan-Scott-Aschbacher type theorem
for semiprimitive groups for example. In \cite{BM} soluble semiprimitive groups were studied and a
classification achieved when the degree is square-free or a product of at most three primes. In Section 2 we
strengthen some of the results of \cite{BM}, related to semiprimitive groups with a regular
normal nilpotent subgroup.
Our first theorem is below. We remark that it is possible to obtain this theorem as a corollary to
Theorem~\ref{mainthm2}, but we offer a separate proof since a result proved along the way (Lemma~\ref{u
normal in balx}) may be useful elsewhere and because it afforded us the opportunity to use such nice results
as \cite[6.4.3]{kurzweilstellmacher} and  \cite{stells4free}.

\begin{thm}
\label{mainthm1}
Let $\Gamma$ be a finite connected graph of valency $d$ and let $G\leqslant \aut{\Gamma}$ be arc-transitive.
Suppose that the local action is semiprimitive with a regular normal nilpotent subgroup and $(d,6)=1$. Then for each vertex $x$ of $\Gamma$
we have $|G_x| \leqslant d!(d-1)!$.
\end{thm}

In fact we arrive at the conclusion of the theorem above by showing that the group $G_{xy}^{[1]}$ is trivial, for some vertex $y$ adjacent to $x$ in $\Gamma$. Here, the group $G_{xy}^{[1]}$ is the kernel of the action of the edge stabiliser $G_{\{x,y\}}$ on the set $\Gamma(x) \cup \Gamma(y)$, it plays a prominent role in our investigations.

Our second theorem is a technical statement about the structure of vertex stabilisers. Weiss in \cite{weiss}
gave a detailed description of the structure of a vertex stabiliser in a locally affine graph. The most
difficult part of Weiss' proof is the application of so-called failure of factorisation arguments. With our
weaker hypothesis it is  difficult even to show that failure of factorisation results may be applied, since
the local action may have more complicated structure from the outset. This is achieved in Lemma~\ref{good
for glauberman} which allows us to employ a result of Glauberman that delivers the following theorem which
in the locally primitive case yields Weiss' result.
  The notation will be explained in Section 3.

\begin{thm}
\label{mainthm2}
Let $\Gamma$ be a finite connected graph and $G \leqslant \aut{\Gamma}$ be arc-transitive. Suppose that the
local
action is semiprimitive with a regular normal nilpotent subgroup and that $\vst{xy}{1}\neq 1$ for some edge $\{x,y\}$ of $\Gamma$. Let $p$ be a
prime dividing $|\vst{xy}{1}|$. Then there is a normal subgroup $L_x \leqslant G_x$ such that for $ V =
\langle \Omega \zent{S} \mid S \in \syl{p}{L_x} \rangle$   and $H := \mathbf J(L_x)
\cent{L_x}{V}/\cent{L_x}{V}$ the following hold:
\begin{itemize}
\item [(a)]  $\vst{xy}{1}$ is a $p$-group and $p\in \{2,3\}$.
\item [(b)] $H = E_1 \times \cdots \times E_r$ and 
$$V = \cent{V}{H} \times [V,E_1] \times \cdots \times [V,E_r].$$
In particular, $E_i$ acts faithfully on $[V,E_i]$ and trivially on $[V,E_j]$ for $j\neq i$.
\item[(c)] $|[V,E_i]| = p^2$ and $E_i \cong \mathrm{SL}_2(p)$ for $i=1,\dots,r$.
\item[(d)] $A = \times_{i=1}^r (A \cap E_i)$ and $|A||\cent{V}{A}| = |V|$ for all $A \in \mathcal A_{V}(H)$.
\end{itemize}
\end{thm}

We also obtain the following corollary to Theorem~\ref{mainthm2} which highlights where possible
counterexamples
to the PSV Conjecture may be lurking.

\begin{cor}
\label{maincor}
Assume the hypothesis of Theorem~\ref{mainthm2} and let $K = G_x^{\Gamma(x)}$ with $R$ the regular normal
nilpotent subgroup of $K$. Then $K$ contains normal subgroups $J$ and $F$ such that $F < R < J$  and one of
the following holds:
\begin{itemize}
\item[(1)] $J/F \cong \sym{3} \times \dots \times \sym{3}$ when $p=2$,
\item[(2)] $J/F \cong \alt{4} \times \dots \times \alt{4}$ when $p=3$.
\end{itemize}
\end{cor}

Below we provide some examples of semiprimitive groups which are either shown to be graph-restrictive by 
Theorem~\ref{mainthm1} or Corollary ~\ref{maincor}, or indicate areas where the PSV Conjecture is currently
unknown.

\begin{eg}
Let $q$ be an odd prime and let $P$ be an abelian $q$-group. Let $H= \mathrm C_2$ act on $P$ by inversion
and let 
$K = P \rtimes H$ act on $K/H$. Since $H$ inverts every nontrivial element of $P$,
Theorem~\ref{thm:spcriterion} shows that 
$K$ is semiprimitive in this action. For $q>3$ Theorem~\ref{mainthm1} shows that $K$ is graph-restrictive.
For
$q=3$ the 
situation is quite different: since $H$ normalises $\Phi(P)$ and inverts $P/\Phi(P)$ we see that every
maximal subgroup $M$
 of $P$ is a normal subgroup of $K$ and that $K/M \cong \sym{3}$. Thus Corollary~\ref{maincor} does not
provide any information
 in this case. If $\Phi(P)=1$ it was shown in \cite{giudicimorgan} that $K$ is graph-restrictive, but for
$\Phi(P) \neq 1$
 this is still an open case of the PSV Conjecture.
\end{eg}

\begin{eg}
Let $q$ be a prime, $a$, $n$ and $m$ integers and let $V=(\mathbb F_{q^a})^n$ be the $n$-dimensional vector
space over $\mathbb F_{q^a}$. Let $H=\mathrm{GL}(V)$ and let $W = V \oplus \dots \oplus V$ be the direct sum
of $m$ copies of $V$. Set $K = W \rtimes H$ (where $H$ acts in the natural way on each copy of $V$) and set
$\Omega = W$, the set of vectors of $W$. Then $K$ acts on $\Omega$ and $W$ is a regular normal abelian
subgroup. Theorem~\ref{thm:spcriterion} shows that $K$ is a semiprimitive group on $\Omega$, and if $q>3$
then this group is graph-restrictive by Corollary~\ref{maincor}.

If $q=3$ then $K$ has no normal section isomorphic to a direct product of groups isomorphic to $\mathrm {Sym}(3)$
unless $n=1$. Then $K$ is just the extension of an elementary abelian $3$-group by an involution which acts by inversion. 
As noted in the previous example, $K$ is graph-restrictive in this case.

If $q=2$ then for $K$ to have a normal section isomorphic to a direct product of groups isomorphic to $\mathrm{Alt}(4)$ we have to have $(n,a)=(2,1)$ or $(n,a)=(1,2)$. If $m=1$ then $K$ is isomorphic to $\mathrm{Sym}(4)$ or $\mathrm{Alt}(4)$ acting naturally on four points, so $K$ is graph-restrictive by \cite{gard}.
If $m>1$ then this is an open case of the PSV Conjecture.
\end{eg}

\begin{eg}
Let $q$ be an odd prime, $m$ an integer and let $E =q_+^{1+2m}$ be an extraspecial group. Since $q$ is odd, 
$\aut{E}$ contains a subgroup $H\cong \mathrm{Sp}_{2m}(q)$. Set $K=E \rtimes H$. Then $H$ acts faithfully on
$E$ and both faithfully and irreducibly on $E/\zent{E}$, so Theorem~\ref{thm:spcriterion} shows that $K$ is
semiprimitive on the cosets of $H$. (These groups were shown to be semiprimitive in \cite[Corollary
4.3]{BM}.)
 For $q>3$ we see that $K$ is graph-restrictive by Theorem~\ref{mainthm1}.
For $q=3$, since the only nontrivial normal subgroups of $K$ contained in $E$ are $\zent{E}$ and $E$ itself,
we see there is no normal subgroup $N$ of $K$ such that either $N$ or $N/\zent{E}$ is isomorphic to a direct
product of groups isomorphic to $\sym{3}$. Hence Corollary~\ref{maincor} shows that $K$ is graph-restrictive
in this case. 
\end{eg}

\begin{eg}
Let $\pi=\{p_1,\dots,p_r\}$ be a finite set of primes such that $p_i \equiv -1 \pmod 3$ for $i=1,\dots,r$
and
$p_1<p_2<\dots<p_r$. For each $p_i \in \pi$ with $p_i >2$ let $V_i$ be an extraspecial group of plus type
and order $p_i^3$. If $p_1=2$ then let $V_1 = \mathrm Q_8$.
Let $H=\langle t \rangle \cong \mathrm C_3$ act on each $V_i$ as an element of order three in
$\aut{V_i}$, note that $H$ acts irreducibly on $V_i/\zent{V_i}$.

Set
$K := (V_1 \times \dots \times V_r) \rtimes H$ and let $K$ act on the cosets of $H$ in $K$. Then $K$ is
semiprimitive by Theorem~\ref{thm:spcriterion} with regular normal nilpotent subgroup $V_1 \times \dots
\times V_r$. If $p_1>2$ then $K$ is graph-restrictive by
Theorem~\ref{mainthm1}. On the other hand, if $p_1=2$ then 
$$K/ (\zent{V_1} \times V_2\times \dots \times V_r) \cong\alt{4},$$
and so Corollary~\ref{maincor} provides no information. 
\end{eg}

\section{Some results on semiprimitive permutation groups}
\label{sec:semiprim}

\begin{defn}
\label{defn:sp}
Let $G$ be a permutation group on $\Omega$. A subgroup $N$ of $G$ is called \emph{semiregular} if $N_\omega
= 1$ for all $\omega \in \Omega$.  We say that $G$ is \emph{semiprimitive} if $G$ is transitive and  every
normal subgroup of $G$ is transitive or semiregular.
\end{defn}

In  \cite[Theorem 3.2]{BM} Bereczky and Mar\'oti give necessary and sufficient conditions for a permutation
group with a
regular normal soluble subgroup to be semiprimitive. In the theorem below we remove the solubility part of
the hypothesis.
   We recall some definitions before stating the result. Suppose
that $G$ is a group with subgroups $H$, $K$ and $L$ such that $K$ and $L$ are normal in $G$ and $L \leqslant
K$. Then $H$ acts on the quotient $K/L$ by $(kL)^h = k^hL$ for $h\in H$. The kernel of this action is the
set $\{ h \in H \mid [K,h] \leqslant L\}$ which is the largest normal subgroup $M$ of $H$ such that $[M,K]
\leqslant L$.

\begin{thm}
\label{thm:spcriterion}
Let $G=K\rtimes H$. Then $G$ is semiprimitive on the cosets of $H$ if and only if $H$ acts faithfully on
every nontrivial $H$-invariant quotient of $K$.
\end{thm}
\begin{proof}
Suppose that $H$ acts faithfully on every $H$-invariant quotient of $K$ and assume that $N \vartriangleleft
G$. If $K\leqslant N$ then $N$ acts transitively on $G/H$, so we may assume that $M:= K \cap N$ is a
proper
subgroup of $K$. Since $N \vartriangleleft G$ we have that $M$ is $H$-invariant, hence $H$ acts on $K/M$.
Since $[N,K] \leqslant N \cap K = M$ we see that $N$ acts trivially on $K/M$, that is, $N \leqslant
\cent{G}{K/M}$. Hence 
$$N \cap H \leqslant \cent{G}{K/M} \cap H = \cent{H}{K/M}=1,$$
and so $N$ is semiregular, as required.

Now assume that $G$ is semiprimitive and let $M$ be a proper normal subgroup of $ K$ which is $H$-invariant.
Suppose that $H$ does not act faithfully on $K/M$. Then $B: =\cent{H}{K/M}$ is a nontrivial normal subgroup
of $H$. Moreover, $K$ normalises $BM$ since by definition $[K,B] \leqslant M$. Hence $BM$ is a normal
subgroup of
$G$. Now $1\neq B \leqslant BM \cap H$ so $BM$ is not semiregular. If $BM$ acts transitively on $G/H$ then
$G=BMH=MH$ and we have that $K=M(K \cap H)=M$, a contradiction.
\end{proof}

Another way to phrase the above criterion is that for every nontrivial normal subgroup $N$ of $H$ we have
$K=[K,N]$.

The following lemma is \cite[Lemma 2.4]{BM} (see also \cite[Corollary 3.1]{BM}), we offer a different proof
which fits with our approach.

\begin{lem}
\label{lem:quot still sp}
Suppose that $G$ is a semiprimitive group with point stabiliser $H$ and $N$ is a normal intransitive
subgroup of $G$. Then the action of $G/N$ on $HN/N$  is faithful and semiprimitive.
\end{lem}
\begin{proof}
We set $\ba{G}=G/N$ and use the bar notation. Let $\ba{K}$ be a normal subgroup of $\ba{G}$. Then $K$ is
normal in $G$, so either $K$
is transitive, which gives $KH=G$ and so $\ba{G}=\ba{K}\ba{H}$ whence $\ba{K}$ is transitive on
$\ba{G}/\ba{H}$, or $K$ is semiregular. The latter implies $K \cap H = 1$, which gives $K \cap HN = N(K \cap
H) = N$, and therefore $\ba{K} \cap \ba{H} = 1$. Hence $\ba{K}$ is semiregular, as required.

To see that $\ba{G}$ is faithful on $\ba{G}/\ba{H}$, let $\ba{C}=\mathrm{core}_{\ba{G}}(\ba{H})$. Then $N
\leqslant C
\leqslant HN$ and $C$ is normal in $G$. If $C$ were transitive, we would have $G=CH \leqslant HN$, which
yields $N$ 
transitive on $G/H$, a contradiction. Thus $C$ is semiregular, so $C=C \cap HN = N (C\cap H) = N$, whence
$\ba{C}=1$ as required.
\end{proof}

It was shown in \cite[Theorem 3.1]{BM} that a soluble semiprimitive group has a unique regular normal subgroup and that this regular normal subgroup contains all semiregular subgroups and is contained in every transitive subgroup. Below we prove a result for semiprimitive groups which are not necessarily soluble, but we assume the existence of a regular normal soluble subgroup.

\begin{thm}
\label{thm:existence of kernels}
Let $G$ be a semiprimitive group with a soluble regular normal subgroup $K$. Then every transitive normal
subgroup of $G$ contains $K$ and every normal semiregular subgroup is contained in $K$. In particular, $K$
is the unique regular normal subgroup.
\end{thm}
\begin{proof}
Let $G$ and $K$ be as in the statement and let $H$ be a point stabiliser in $G$.
We first prove that $\cent{G}{K} \leqslant K$.
Indeed, assume that $K \neq \cent{G}{K} K$ and let $S \leqslant H$ be such that $KS = K \cent{G}{K}$. 
Then $S$ is normal in $H$ and the solubility of $K$ gives $[K,S] \leqslant [K,K] \neq  K$.
Now Theorem~\ref{thm:spcriterion} shows that $S=1$, a contradiction.

Now we assume that the theorem is false and let $G$ be a counter-example of minimal order, so that $G$ is a
semiprimitive group which has a soluble regular normal subgroup $K$ and $T$ is another  transitive normal
subgroup. Set $M=T \cap K$, then since $G$ is a counter-example to the theorem, $T \cap K \neq K$ so $M$ is
a proper subgroup of $K$. By the previous paragraph $[T,K] \neq 1$. Moreover
$[T,K] \leqslant M < K$ so $1 \neq M$ is intransitive and semiregular. Now Lemma \ref{lem:quot still sp}
shows that
$G/M$ acts semiprimitively  on the cosets of $HM/M$ where $H$ is a point stabiliser in $G$.  Since $THM =
TH=G$, $T/M$ is transitive on the cosets of $H/M$ in $G/M$ and $K \cap HM=M(K \cap H) = M$, so $K/M$ is a
regular normal soluble subgroup of $G/M$.  Since $|G/M| < |G|$, $G/M$ is not a counter-example to the
theorem, so $T/M$ contains $K/ M$. This implies $T$ contains $K$, a contradiction to $G$ being a
counter-example.

The second case is similar and is omitted.
\end{proof}

The following example shows that the conclusion of the above theorem is false without the solubility
hypothesis. Set 
$$G=( T_1\times T_2\times T_3 ) : \langle x \rangle$$
where $T_1 \cong T_2 \cong T_3 \cong \alt{5}$ and $x$ is an involution that acts as an outer automorphism on each
$T_i$ for $i=1,2,3$.
 Let $D$ be a full diagonal subgroup of $T_1 \times T_2 \times T_3$ which is normalised by $x$ and let $H=\langle
D, x \rangle$.
 Let $G$ act on $G/H$ and note that $T_1T_2$, $T_2T_3$ and $T_1T_3$ are three distinct
regular normal subgroups. Since $G= T_1T_2 \rtimes H$ and $H$ acts faithfully on $T_1T_2$, $T_1T_2 / T_1$
and on $T_1T_2/T_2$, $G$ is semiprimitive. 

\begin{lem}
\label{lem: fst(g) eq f(k)}
Let $G$ be a semiprimitive group with a soluble regular normal subgroup $K$. Then every normal nilpotent
subgroup is contained in $K$. In particular, $\mathbf F(G)=\mathbf F( {K})$.
\end{lem}
\begin{proof}
Suppose that $N$ is a normal nilpotent subgroup of $G$ not contained in $K$. Since $G$ is semiprimitive
either $N$ is semiregular or $N$ is transitive. Theorem~\ref{thm:existence of kernels} shows that  $K
\leqslant N$ and $N$ is transitive. Since $N$ is nilpotent we have $[K,N] < K$, then $N \cap H \leqslant
\cent{H}{K/[K,N]}=1$ by Theorem~\ref{thm:spcriterion}. Thus $N$ is regular and $N=K$.
\end{proof}

We note that our next lemma generalises \cite[Theorem 3.4]{BM}.

\begin{lem}
\label{lem: fit(H) and K coprime orders}
Let $G$ be a semiprimitive group with point stabiliser $H$ and regular normal nilpotent subgroup $K$. If
there is a prime $p$ such that $\op{p}{H} \neq 1$, then $p$ does not divide $|K|$.
\end{lem}
\begin{proof}
Assume $p$ divides $|K|$ and note that $N=\op{p'}{K}$ is a proper semiregular subgroup of $K$. By
Lemma~\ref{lem:quot still sp} the group
$G/N$ is semiprimitive with point stabiliser $HN/N$ and regular normal subgroup $K/N$. Now $K/N$ and
$\op{p}{H}K/N$ are normal $p$-subgroups of $G/N$ so Lemma~\ref{lem: fst(g) eq f(k)} implies that
$\op{p}{H}K/N = K/N$. This implies $\op{p}{H} \leqslant K \cap H = 1$, a contradiction.
\end{proof}

\section{Locally semiprimitive graphs with regular normal nilpotent subgroups}

In this section we assume that $G$ is a group acting faithfully and vertex-transitively on a connected
finite graph $\Gamma$. Moreover we assume that the local action is semiprimitive with a regular normal
nilpotent subgroup. We will prove Theorems~\ref{mainthm1} and \ref{mainthm2}.

If $A$ is a set of vertices of $\Gamma$ and $H \leqslant G_A$, we write $H^A$ for the permutation group 
induced on $A$ by $H$, in most cases $A$ will be the neighbourhood of some vertex.
  We fix an edge $e=\{x,y\}$ of $\Gamma$ and begin by assuming that $\vst{xy}{1} \neq 1$. 

We will sometimes use the following two results without reference.

\begin{prop}
Let $R$ be a subgroup of $G_x$ and suppose that $R^{\Gamma(x)}$ is semiregular. Then $R \cap G_{xy}
\leqslant \vstx{1}$.
\end{prop}
\begin{proof}
We have $R^{\Gamma(x)} \cap (G_{xy})^{\Gamma(x)}=1$, that is $R\vstx{1} \cap G_{xy} \leqslant G_x^{[1]}$.
The Dedekind identity gives
$$(R \cap G_{xy}) \vstx{1} = R\vstx{1} \cap G_{xy} \leqslant \vstx{1}$$
and so $R\cap G_{xy} \leqslant \vstx{1}$ as required.
\end{proof}

\begin{lem}
\label{lem:fundlem}
Suppose that $K\leqslant  G_x \cap G_y$. If either (a) or (b) below hold, then $K=1$.
\begin{itemize}
\item[(a)] The groups $\norm{G_x}{K}^{\Gamma(x)}$ and $\norm{G_y}{K}^{\Gamma(y)}$ are transitive.
\item[(b)] The group $\norm{G_x}{K}^{\Gamma(x)}$ is transitive and $\norm{G_e}{K} \not\leqslant G_x \cap G_y$.
\end{itemize}
\end{lem}
\begin{proof}
Suppose that (a) holds and set $H = \grp{ \norm{G_x}{K},\norm{G_y}{K}} \leqslant \norm{G}{K}$. Since
$\Gamma$ is connected, $H$ acts edge-transitively. Let $u$ be any vertex of $\Gamma$ and let $v$ be adjacent
to $u$. Then there exists $h\in H$ such that $\{x,y\}^h = \{u,v\}$. Now we obtain 
$$K = K^h \leqslant  (G_x \cap G_y)^h = G_u \cap G_ v \leqslant  G_{u}$$
whence $K$ fixes every vertex of $\Gamma$, and therefore $K=1$. The case that (b) holds is similar and is
omitted.
\end{proof}

\begin{lem}
\label{lem:th-w}
There exists a prime $p$ such that $\vst{xy}{1}$, $\fst{\vst{x}{1}}$ and $\fst{G_{xy}}$ are $p$-groups. In
particular $\vst{xy}{1} \leqslant \op{p}{\vst{x}{1}}$.
\end{lem}
\begin{proof}
This follows from \cite[Corollary 3]{spigath-w} and the fact that  $\fst{\vstx{1}} \leqslant \fst{G_{xy}}$. 
\end{proof}

We now establish some notation that will hold for the remainder of the paper. Recall that, for a $p$-group
$P$, $\Omega\zent{P}$ is the subgroup of $\zent{P}$ generated by the elements of order $p$. We set
\begin{eqnarray*}
Q_x		& 	=	&	\op{p}{\vst{x}{1}}, \\
L_x 	&	= 	&	\langle (Q_xQ_y) ^{G_x} \rangle, \\
R_0 	&	= 	&	\op{p'}{L_x}, \\
Z_{xy} 	&	=	& 	 \Omega \zent{Q_xQ_y}, \\
Z_x 		& 	=	&  \grp{Z_{xy}^{ G_x }}.\\
\end{eqnarray*}

\begin{lem}
\label{firstlem}
The following hold:
\begin{itemize}
\item[(i)] $Q_y^{\Gamma(x)} \neq 1$;
\item[(ii)] $L_x$ is transitive on $\Gamma(x)$.
\end{itemize}
\end{lem}
\begin{proof}
If (i) is false then $Q_y \leqslant \vstx{1}$ and so $Q_y=Q_x$. From Lemma~\ref{lem:fundlem} it follows that
$Q_x = Q_y = 1$, a contradiction since $\vst{xy}{1} \leqslant Q_x \cap Q_y$ by Lemma \ref{lem:th-w}. Clearly
$L_x$ is normal in $G_x$ and by part (i) we have that $1 \neq Q_y^{\Gamma(x)} \leqslant L_x^{\Gamma(x)} \cap
G_{xy}^{\Gamma(x)}$. Hence $L_x$ is not semiregular, so $L_x^{\Gamma(x)}$ is transitive.
\end{proof}

By Theorem~\ref{thm:existence of kernels},  $L_x^{\Gamma(x)}$ contains the nilpotent regular normal subgroup
of $G_x^{\Gamma(x)}$. We let $R$ be the full pre-image of this subgroup, so  $L_x \cap \vstx{1} \leqslant R
\leqslant L_x$ and $R^{\Gamma(x)}$ is the nilpotent normal regular subgroup of $G_x^{\Gamma(x)}$.

\begin{lem}
\label{nilp reg nml coprime to p}
The order of $R^{\Gamma(x)}$ is coprime to $p$.
\end{lem}
\begin{proof}
We have  $Q_y^{\Gamma(x)}$ is a nontrivial normal subgroup of $G_{xy}^{\Gamma(x)}$ so the lemma follows from
Lemma \ref{lem: fit(H) and K coprime orders}. 
\end{proof}

\begin{lem}
\label{opgx=opgx1} We have $\op{p}{G_x}=Q_x$, $\cent{G_x}{Q_x}^{\Gamma(x)}$ is intransitive and 
$$\cent{G_x}{Q_x} = \zent{Q_x} \op{p'}{G_x}.$$ \end{lem}
\begin{proof}
We see that $\op{p}{G_x}^{\Gamma(x)}$ is a nilpotent normal subgroup of $G_x^{\Gamma(x)}$ so Lemma~\ref{nilp
reg nml coprime to p} and Lemma~\ref{lem: fst(g) eq f(k)} show that $\op{p}{G_x}^{\Gamma(x)}=1$. This gives
$\op{p}{G_x} \leqslant \vstx{1}$ from which the first part of the lemma follows. For the third part, we
just need to show that $\cent{G_x}{Q_x} \leqslant \zent{Q_x}\op{p'}{G_x}$ since the reverse inclusion is
obvious. 

Since $\vst{xy}{1}$ is nontrivial and is contained in $Q_x \cap Q_y$, it follows from Lemma~\ref{lem:fundlem} (b) that 
$\cent{G_x}{Q_x}^{\Gamma(x)}$ is an intransitive normal  subgroup of $G_x^{\Gamma(x)}$. Theorem~\ref{thm:existence of kernels} implies that
$\cent{G_x}{Q_x}^{\Gamma(x)} \leqslant R^{\Gamma(x)}$ and Lemma~\ref{nilp reg nml coprime to p} shows that a
Sylow $p$-subgroup $P$ of $\cent{G_x}{Q_x}$ is contained in $\vstx{1}$, whence $P \leqslant
\cent{\vstx{1}}{Q_x} = \zent{Q_x}$ (by Lemma~\ref{lem:th-w}). Now we see that $|\cent{G_x}{Q_x}  :
\zent{Q_x} | $ is coprime to $p$, so the Schur-Zassenhaus Theorem gives  $D \leqslant \cent{G_x}{Q_x}$ such
that $\cent{G_x}{Q_x} = \zent{Q_x} D \cong \zent{Q_x} \times D$. Now $D=\op{p'}{\cent{G_x}{Q_x}} \leqslant
\op{p'}{G_x}$ and we are done.
\end{proof}

\begin{lem}
\label{lem:lx/qx centre}
We have $[L_x,\vstx{1}] \leqslant Q_x$, in particular, $ L_x \cap \vstx{1}/Q_x \leqslant \zent{L_x / Q_x}$.
\end{lem}
\begin{proof}
Since $\vstx{1} \leqslant G_y$ we see that $\vstx{1}$ normalises $Q_y$. Thus  
$$[L_x,\vstx{1} ] = [ \langle (Q_x Q_y) ^{G_x} \rangle, \vstx{1}] = \langle [ Q_xQ_y,\vstx{1}]^{G_x}
\rangle.$$
 Now $Q_x$ and $Q_y$ normalise each other and $[\vstx{1},Q_y] \leqslant \vstx{1} \cap Q_y \leqslant Q_x$, so
 $$[Q_xQ_y,\vstx{1}] \leqslant Q_x.$$
This shows that $L_x \cap \vstx{1}/Q_x$ is contained in the centre of $L_x/Q_x$.
\end{proof}

\begin{lem}
\label{lx n gx1 : qx coprime to p}
The group $Q_x$ is the Sylow $p$-subgroup of $R$.
 \end{lem}
\begin{proof}
By Lemma~\ref{lem:lx/qx centre} we see that $L_x \cap \vstx{1} / Q_x$ is  abelian. Since $\op{p}{L_x \cap
\vstx{1}} \leqslant \op{p}{\vstx{1}} = Q_x$, we see that $p$ does not divide $|L_x \cap
\vstx{1}:Q_x|$. Since 
$$|R:Q_x| = |R:L_x \cap \vstx{1}||L_x \cap \vstx{1} :Q _x|$$ the result follows from Lemma~\ref{nilp reg nml
coprime to p}.
\end{proof}

\begin{lem}
\label{qx a sylow of clxzx}
We have $Z_x \leqslant \Omega\zent{Q_x}$ and $Q_x$ is the Sylow $p$-subgroup of $\cent{L_x}{Z_x}$.
\end{lem}
\begin{proof}
Note that $[Z_{xy},Q_x] \leqslant [Z_{xy},Q_xQ_y] = 1$, so $Z_{xy} \leqslant \cent{L_x}{Q_x}$. In
particular, $Z_{xy}$ is contained in a Sylow $p$-subgroup of $\cent{G_x}{Q_x}$, which by
Lemma~\ref{opgx=opgx1} is equal to $\zent{Q_x}$. Since $Z_{xy}$ is elementary abelian, we have $Z_{xy}
\leqslant \Omega\zent{Q_x}$ and from this it follows that $Z_x \leqslant \Omega\zent{Q_x}$.

For the second part we have that $Q_x
\leqslant \cent{L_x}{Z_x}$. If $\cent{L_x}{Z_x} \leqslant L_x \cap \vstx{1}$ then we are done by
Lemma~\ref{lx n gx1 : qx coprime to p}. Otherwise, we
see that $\cent{L_x}{Z_x}$ is a normal subgroup of $G_x$ which is not contained in $\vstx{1}$. Since $1\neq
Z_{xy}$ is centralised by $\cent{L_x}{Z_x}$ we have that $\cent{L_x}{Z_x}$ is semiregular on $\Gamma(x)$.
Then $\cent{L_x}{Z_x}( L_x \cap \vstx{1}) \leqslant R$ and so the result follows from Lemma \ref{lx n gx1 :
qx coprime to p}.
\end{proof}

Recall that a group $X$ is $p$-separable if the series 
$$1 \leqslant \mathrm O_p(X) \leqslant \mathrm O_{p p'}(X) \leqslant \mathrm O_{pp'p} (X) \leqslant \cdots$$
 terminates with $X$. Here the group $\mathrm O_{pp'}(X)$ is defined by
$$\mathrm O_{pp'}(X) / \mathrm O_p(X) = \mathrm O_{p'}(X/\mathrm O_p(X))$$
and the other groups in the series are defined in the same recursive manner.
A soluble group is $p$-separable for all primes $p$.

\begin{lem}
\label{properties of lx}
The following hold:
\begin{itemize}
\item[(i)] $L_x = R Q_y$;
\item[(ii)]  $Q_xQ_y \in \syl{p}{L_x}$;
\item[(iii)]  $L_x$ is $p$-separable;
\item[(iv)]  if $r$ is a prime dividing $|L_x \cap \vstx{1} : Q_x|$ then $r$ divides $|R:R \cap \vstx{1}|$.
\end{itemize}
\end{lem}
\begin{proof}
Since $R$ is transitive on $\Gamma(x)$ we have that $\{ Q_y ^{G_x} \} = \{Q_y ^{R} \}$. Whence 
$$L_x = \langle (Q_xQ_y)^{G_x} \rangle  = \langle (Q_xQ_y)^R \rangle  \leqslant RQ_x Q_y = RQ_y$$
and since $RQ_y \leqslant L_x$ we have equality so (i) holds.
By Lemma~\ref{lx n gx1 : qx coprime to p} we have $R \cap Q_y = Q_x \cap Q_y$ and it follows that  $Q_xQ_y
\in \syl{p}{L_x}$ which is (ii).
 We observe that  $ Q_x = \op{p}{L_x}$, $R=\op{pp'}{L_x}$ and $L_x =
\op{pp'p}{L_x}$ which gives (iii).

Finally  suppose that $r$ is a prime dividing $|L_x \cap \vstx{1} : Q_x|$ and let $E$ be a Sylow
$r$-subgroup of $L_x \cap \vstx{1}$.
Let $\ba{L_x}  = L_x/ Q_x$, then $\ba{E}$ is a nontrivial central subgroup of $\ba{L_x}$ by Lemma
\ref{lem:lx/qx centre}. If $r$ does not divide $|R: L_x \cap \vstx{1}|$, then by part (1) $\ba{E}$ is a Sylow
$r$-subgroup of $\ba{L_x}$, and so there is a normal complement $\ba{F}$ to $\ba{E}$ in $\ba{L_x}$ by 
Burnside's Normal $p$-Complement Theorem. We have  $\ba{Q_xQ_y} \leqslant \ba{F}$ and therefore $\ba{L_x} = \langle
\ba{Q_xQ_y}^{\ba{L_x}} \rangle \leqslant \ba{F}$.
Now $\ba{E} \leqslant \ba{E} \cap \ba{F} = 1$, a contradiction.
\end{proof}

The next lemma implies that $\op{p}{L_x/\op{p'}{L_x}}= \op{p}{L_x} \op{p'}{L_x}/\op{p'}{L_x}$. This is not
standard behaviour for $p$-separable groups, indeed, the group $X:=\sym{3} \times \cyc{2}$ is $2$-separable,
but $$\op{2}{X/\op{3}{X}}=X/\op{3}{X} \neq \op{2}{X}\op{3}{X}/\op{3}{X}.$$

\begin{lem}
\label{u normal in balx}
Let $\ba{L_x} = L_x / R_0$. Then  $\op{p}{\ba{L_x}} = \ba{Q_x}$.
\end{lem}
\begin{proof}
Let $U \leqslant L_x$ be such that $\ba{U}=\op{p}{\ba{L_x}}$ and note that $\ba{Q_x} \leqslant \ba{U}$ so
$Q_x \leqslant U$.
 Choose $U_0$ to be a Sylow $p$-subgroup of $U$ so that $U = R_0 U_0$ and we may
assume that $U_0 \leqslant Q_xQ_y$ by Lemma \ref{properties of lx}. Note that $Q_x \leqslant U_0$ and (since
$R_0 \leqslant R$) we have
$$ U \cap R = R_0U_0 \cap R = R_0(U_0 \cap R) = R_0 Q_x,$$
where the last equality is by Lemma~\ref{lx n gx1 : qx coprime to p}.
Now $[U_0,R] \leqslant [U,R] \leqslant U \cap R = R_0Q_x < R$ where the last inequality holds since Lemma~\ref{opgx=opgx1} shows that $R_0Q_x$ is intransitive but $R$ is transitive by definition. 
 By our choice of $R$, $R^{\Gamma(x)}$ is the normal nilpotent regular subgroup of $(G_x)^{\Gamma(x)}$. Then
$$[(U_0)^{\Gamma(x)},R^{\Gamma(x)}]=[U_0,R]^{\Gamma(x)} \leqslant (Q_xR_0)^{\Gamma(x)} =  (R_0) ^{\Gamma(x)}
< R^{\Gamma(x)}.$$
The calculation above shows that $(U_0)^{\Gamma(x)}$ centralises the nontrivial quotient $R^{\Gamma(x)}/
(R_0)^{\Gamma(x)}$. On the other hand,  since $R_0$ is normal in $G_x$, $ (R_0) ^{\Gamma(x)}$  is a $
(G_{xy}) ^{\Gamma(x)}$-invariant normal subgroup of $R^{\Gamma(x)}$ and Theorem~\ref{thm:spcriterion} says
that $(G_{xy})^{\Gamma(x)}$ acts faithfully on $R^{\Gamma(x)}/ (R_0)^{\Gamma(x)}$. Hence
$(U_0)^{\Gamma(x)}=1$, that is, $U_0 \leqslant L_x \cap \vstx{1}$ and therefore
$U_0 \leqslant Q_x$. Hence $\ba{U} \leqslant \ba {Q_x}$ as required.
  \end{proof}

\begin{lem}
\label{lem: p eq 2,3 and q}
We have $p \in \{2,3\}$ and $q:=5-p$ divides $|R/L_x \cap \vstx{1}|$.
\end{lem}
\begin{proof}
Let $\ba{L_x} = L_x / R_0$ and note that by \cite[6.4.3]{kurzweilstellmacher} $\ba{L_x}$ has characteristic
$p$. If  $p\geq 5$ then $\ba{L_x}$ is $p$-stable by \cite[9.4.5 (1)]{kurzweilstellmacher}. If $p=3$ and $q
\nmid |R/L_x \cap \vstx{1}|$ then $\ba{L_x}$ has odd order by Lemma~\ref{properties of lx} (iv), and is
therefore $p$-stable by \cite[9.4.5
(2)]{kurzweilstellmacher}. If $p=2$ and $q \nmid |R / L_x \cap \vstx{1}|$ then $\ba{L_x}$ has order coprime
to three by Lemma~\ref{properties of lx} and is therefore
$\sym{4}$-free.

Suppose for a contradiction that one of the first two cases holds. Then we may apply \cite[9.4.4
(b)]{kurzweilstellmacher} to $\ba{L_x}$ to obtain a nontrivial characteristic subgroup $\ba{D}$ of
$\ba{Q_xQ_y}$ which is normal in $\ba{L_x}$. Lemma~\ref{u normal in balx} gives $\ba{D} \leqslant \ba{Q_x}$.
Since
the preimage of $\ba{Q_x}$ is $Q_x R_0 \cong Q_x \times R_0$ we may choose a subgroup $D$ of $Q_x$ such that
$D$ has image $\ba{D}$. Since $\ba{D}$ is normal in $\ba{L_x}$ we see that $DR_0$ is normal in $L_x$ and $D
\in \syl{p}{DR_0}$. The Frattini argument shows that $L_x = DR_0 \norm{L_x}{D}$. Since $[R_0,D] \leqslant
[R_0,Q_x] = 1$ we have that $D$ is normal in $L_x$. Now $\ba{Q_xQ_y}$ is isomorphic to $Q_xQ_y$, so $D$ is
characteristic in $Q_xQ_y$. But now $1 \neq D$ is normalised by $\langle L_x, G_{\{x,y\}} \rangle$, a
contradiction.

Suppose now that the third case holds. Since $\ba{L_x}$ is now $\sym{4}$-free, we may use \cite{stells4free}
in place of \cite[9.4.4 (b)]{kurzweilstellmacher} and a similar argument to above leads to a contradiction. 
\end{proof}

\begin{proof}[Proof of Theorem~\ref{mainthm1}]
Let $\Gamma$ and $G$ be as in the hypothesis of Theorem~\ref{mainthm1}. Since the index of $\vst{xy}{1}$ in $G_x$ is at most $d!(d-1)!$ we assume for a contradiction that $\vst{xy}{1}\neq 1$.
In particular, we may apply all of the results in this section. Let $N^{\Gamma(x)}$ be the regular normal nilpotent
subgroup
of $(G_x)^{\Gamma(x)}$, and note that $d=|N^{\Gamma(x)}|$ is coprime to six. Using the notation of
Lemma~\ref{lem: p
eq 2,3 and q} we have  
$$N^{\Gamma(x)} \cong R/ (L_x \cap \vstx{1}).$$
 Lemma~\ref{lem: p eq 2,3 and q} shows that either $2 \mid | N^{\Gamma(x)}|$ or $3 \mid | N^{\Gamma(x)}|$, a
contradiction.
\end{proof}

In the next proposition we use the Thompson subgroup. For a $p$-group $S$ we let
$\mathcal A(S)$ be the set of elementary abelian subgroups of maximal order in $S$. Then the \emph{Thompson
subgroup of} $S$ is
$$\mathbf J(S) = \langle A \mid A \in \mathcal A(S) \rangle.$$
For a group $F$ and a prime $p$ we set 
$$\mathbf J(F) = \langle \mathbf J(S) \mid S \in \syl{p}{F} \rangle.$$

\begin{prop}
\label{jlx bar}
With $\ba{L_x} = L_x / R_0$ we have $\mathbf J(\ba{L_x}) = \ba{\mathbf J(L_x)}$.
\end{prop}
\begin{proof}
Let $S$ be a Sylow $p$-subgroup of $L_x$. Since $R_0$ has order coprime to $p$ we see that $\ba{S} \in
\syl{p}{\ba{L_x}}$ and $\mathbf J(\ba{S}) = \ba{ \mathbf J(S)}$. Hence
$$\mathbf J(\ba{L_x}) = \langle \mathbf J(\ba{S}) \mid \ba{S} \in \syl{p}{\ba{L_x}} \rangle= \langle
\ba{\mathbf J(S)} \mid S \in \syl{p}{L_x} \rangle = \ba{ \mathbf J(L_x)}.$$
\end{proof}

Following \cite[§9.2]{kurzweilstellmacher} we say that a group $F$ is Thompson factorizable with respect to
the prime $p$ if $p$ divides $|F|$ and for some Sylow $p$-subgroup $S$ of $F$ we have 
$$F=\op{p'}{F}\cent{F}{\Omega\zent{S}}\norm{F}{\mathbf J(S)}.$$
We use this notion in the next lemma.

\begin{lem}
\label{good for glauberman}
Let $\ba{L_x}=L_x / R_0$. Then $\op{p'}{\ba{L_x}} = 1$ and $\ba{L_x}$ is not Thompson factorizable with
respect to $p$.
\end{lem}
\begin{proof}
Since $R_0= \op{p'}{L_x}$ the first part of the claim holds trivially. Suppose the latter part of the claim
is false and set 
$$\ba{V} = \langle \Omega\zent{\ba S} \mid \ba S \in \syl{p}{\ba {L_x}} \rangle.$$
By our assumption that $\ba{L_x}$ is Thompson factorizable, \cite[9.2.12]{kurzweilstellmacher} implies
 $$\mathbf J(\ba{L_x}) \leqslant \cent{\ba{L_x}}{\ba{V}}.$$
Using Lemmas \ref{firstlem} and \ref{properties of lx} (ii) and the fact that $|R_0|$ is coprime to $p$ we
have  
$$\ba{V} = \langle \Omega\zent{\ba{Q_xQ_y}}^{\ba{L_x}} \rangle= \ba{\langle \Omega \zent{ Q_xQ_y}^{L_x}
\rangle} = \ba{Z_x}.$$
By Proposition~\ref{jlx bar} we have $\mathbf J(\ba{L_x}) = \ba{\mathbf J(L_x)}$ and so $\mathbf
J(\ba{Q_xQ_y})=\ba{\mathbf J(Q_xQ_y)} \leqslant \cent{\ba{L_x}}{\ba{Z_x}}$. Again using that $|R_0|$ is
coprime to $p$, we have that $\cent{\ba{L_x}}{\ba{Z_x}} = \ba{\cent{L_x}{Z_x}}$. Hence $\ba{\mathbf
J(Q_xQ_y)} \leqslant \ba{\cent{L_x}{Z_x}}$ implies that 
 $$\mathrm R_0 \mathbf J(Q_xQ_y) \leqslant R_0\cent{L_x}{Z_x} = \cent{L_x}{Z_x}.$$
By Lemma \ref{qx a sylow of clxzx}, $Q_x$ is a Sylow $p$-subgroup of $ \cent{L_x}{Z_x}$, whence $\mathbf
J(Q_xQ_y) \leqslant Q_x$. We obtain $\mathbf J(Q_xQ_y) = \mathbf J(Q_x)$,  a contradiction.
\end{proof}

Let $J_x = \mathbf J(L_x) \cent{L_x}{Z_x}$. Note that $Q_xR_0 \leqslant \cent{L_x}{Z_x} \leqslant J_x $.

\begin{prop}
\label{jx trans}
The group $J_x^{\Gamma(x)}$ is transitive.
\end{prop}
\begin{proof}
Clearly $J_x$ is  normal in $G_x$, so if the result is false then $J_x$ is semiregular on $\Gamma(x)$. Then
$\mathbf J(Q_xQ_y) \leqslant J_x \cap G_{xy} \leqslant \vstx{1}$ and from this it follows that $\mathbf
J(Q_xQ_y) = \mathbf J(Q_x)$, a contradiction.
\end{proof}

We can now prove Theorem~\ref{mainthm2} which we restate for convenience and make the substitution $Z_x =
V$. 

\smallskip

\noindent\textbf{Theorem~\ref{mainthm2}.}\emph{
Let  $H= J_x / \cent{ L_x }{ Z_x }$. Then the following hold:
\begin{itemize}
\item [(a)] $p\in \{2,3\}$.
\item [(b)] $H = E_1 \times \cdots \times E_r$ and 
$$Z_x = \cent{Z_x}{H} \times [Z_x,E_1] \times \cdots \times [Z_x,E_r].$$
In particular, $E_i$ acts faithfully on $[Z_x,E_i]$ and trivially on $[Z_x,E_j]$ for $j\neq i$.
\item[(c)] $|[Z_x,E_i]| = p^2$ and $E_i \cong \mathrm{SL}_2(p)$ for $i=1,\dots,r$.
\item[(d)] $A = \times_{i=1}^r (A \cap E_i)$ and $|A||\cent{Z_x}{A}| = |Z_x|$ for all $A \in \mathcal
A_{Z_x}(H)$.
\end{itemize}
}
\begin{proof}
By Lemma \ref{good for glauberman} we may apply \cite[9.3.8]{kurzweilstellmacher} to $\ba{L_x} = L_x / R_0$
which
yields statements (a)-(d) for $\mathbf J (\ba{L_x})\cent{\ba{L_x}}{\ba{Z_x}}/\cent{\ba{L_x}}{Z_x}$. By
Proposition~\ref{jlx bar} we have $\mathbf J(\ba{L_x}) = \ba{\mathbf J(L_x)}$ and since $|R_0|$ is coprime
to $p$ we have $\cent{\ba{L_x}}{\ba{Z_x}} = \ba{\cent{L_x}{Z_x}}$. Furthermore, since $[R_0,Z_x] \leqslant
[R_0,Q_x] = 1$ we see $R_0 \leqslant \cent{L_x}{Z_x}$. Hence 
$$\mathbf J(\ba{L_x})\cent{\ba{L_x}}{\ba{Z_x}}/\cent{\ba{L_x}}{\ba{Z_x}} = \ba{\mathbf J(L_x)
\cent{L_x}{Z_x}}/\ba{\cent{L_x}{Z_x}} \cong \mathbf J(L_x) \cent{L_x}{Z_x} / \cent{L_x}{Z_x}$$
and so the statement above holds.
\end{proof}

\begin{proof}[Proof of Corollary~\ref{maincor}]
Theorem~\ref{mainthm2} gives information about $J_x / \cent{L_x}{Z_x}$. We now convert this into information
about $J_x / (J_x \cap \vstx{1}) \cong J_x^{\Gamma(x)}$. 
 Write $J_x^{[1]}:=J_x \cap \vstx{1}$, $M_x := \cent{L_x}{Z_x}$ and note that $J_x / J_x^{[1]}$ contains the
normal subgroup $M_x J_x^{[1]} / J_x^{[1]}$.  Since $J_x^{\Gamma(x)}$ is transitive by Proposition \ref{jx
trans}, we may choose a subgroup $R$ of $J_x$ so that $J_x^{[1]} \leqslant R \leqslant J_x$ and
$R^{\Gamma(x)}$ is the nilpotent regular normal subgroup of $G_x$. Since $M_x^{\Gamma(x)}$ is intransitive,
we have $M_x \leqslant M_x J_x^{[1]} \leqslant R$.
 The quotient $J_x/M_xJ_x^{[1]}$ is visible as a quotient of 
 $$J_x / M_x \cong E_1 \times \dots \times E_r$$ where each $E_i$ is isomorphic to $\mathrm{SL}_2(p)$ with
$p=2$ or $p=3$.
 Lemma~\ref{lem:lx/qx centre} yields $[J_x,J_x^{[1]}] \leqslant [L_x, L_x\cap \vstx{1}] \leqslant Q_x$, so
we obtain 
$$[J_x, M_x J_x^{[1]}] \leqslant [J_x,M_x] [J_x,J_x^{[1]}] \leqslant M_x Q_x \leqslant M_x,$$
whence $M_x J_x^{[1]}/M_x \leqslant \zent{J_x/ M_x}$. Choose $F \leqslant J_x$ with $M \leqslant F$
so that $F/M_x = \zent{J_x/M_x}$. Then 
$$J_x^{[1]} \leqslant M_x J_x^{[1]} \leqslant F \leqslant J_x,$$
$F$ is normal in $G_x$ and $J_x/ F$ is isomorphic to a direct product of groups isomorphic to
$\mathrm{PSL}_2(p)$ where  $p=2$ or $p=3$. This gives the corollary.
\end{proof}


\begin{thebibliography}{99}
 
\bibitem{BM}
\'A. Bereczky, A. Mar\'oti, On groups with every normal subgroup transitive or semiregular, J. Algebra 319
(2008) 1733--1751.

\bibitem{gard}
A. Gardiner,  Arc-transitivity in graphs, Quart. J. Math. Oxford 24 (1973) 399--407.

\bibitem{giudicimorgan}
M. Giudici, L. Morgan, A class of semiprimitive groups that are graph-restrictive. To appear in Bull. London Math. Soc. 
doi: 10.1112/blms/bdu076.


\bibitem{kurzweilstellmacher}
H. Kurzweil,  B. Stellmacher, The theory of finite groups. 
An introduction.  Universitext. Springer-Verlag, New York, 2004.
 
 \bibitem{PSV}
P. Poto\v{c}nik, P. Spiga, G. Verret, On graph-restrictive permutation groups. J. Combin. Theory Ser. B 102
(2012),  820--831.
 
\bibitem{spigath-w}
P. Spiga, Two local conditions on the vertex stabiliser of arc-transitive graphs and their effect on the
Sylow subgroups. J. Group Theory 15 (2012),   23--35.

\bibitem{spigaverretintrans}
P. Spiga, G. Verret, On intransitive graph-restrictive groups. J. Algebr. Comb.  {40} (2014), 179--185.

\bibitem{stells4free} B. Stellmacher,  A characteristic subgroup of {$\Sigma_4$}-free groups. Israel Journal
of Mathematics. 94 (1996), 367--379.

\bibitem{trofweiss1}
V. I. Trofimov, R. M. Weiss, Graphs with a locally linear group of automorphisms. Math. Proc. Cambridge
Philos. Soc. 118 (1995),   191--206. 

\bibitem{verret}
G. Verret, On the order of arc-stabilisers in arc-transitive graphs. Bull. Aust. Math. Soc. 80 (2009),
498--505.

\bibitem{weisscon}
R. Weiss, s-transitive graphs. Colloq. Math. Soc. J\'anos Bolyai 25 (1978) 827--847.

\bibitem{weiss} R. Weiss, An application of {$p$}-factorization methods to symmetric graphs.  Math. Proc.
Cambridge Philos. Soc.  85 (1979),  43--48.



\end{thebibliography}
\end{document}